\documentclass[12pt]{amsart}

\usepackage{amssymb}

\theoremstyle{plain}
  \newtheorem{theorem}{Theorem}[section]
  \newtheorem{proposition}[theorem]{Proposition}
  \newtheorem{lemma}[theorem]{Lemma}
  
\theoremstyle{definition}

\newtheorem{conjecture}[theorem]{Conjecture}
\newtheorem{observation}[theorem]{Observation}
\newcommand{\Bier}{{\mathrm{Bier}}}
\begin{document}

\title{Nonpolytopal nonsimplicial lattice spheres with nonnegative toric $g$-vector}

\author{Louis J.\ Billera}
\address{Department of mathematics\\
      Cornell University\\
      Ithaca, NY 14853-4201, USA}
\email{billera@math.cornell.edu}

\author{Eran Nevo}
\address{Department of Mathematics\\
    Ben-Gurion University of the Negev\\
    Be'er-Sheva 84105, Israel}
\email{nevoe@math.bgu.ac.il}
\thanks{Research of the first author was partially supported by NSF grant DMS-0555268; that of the second author was partially supported by NSF grant DMS-0757828 and by Marie Curie grant IRG-270923.}

\keywords{toric $g$-vector, polytopes, Bier poset, multiplex}

\maketitle

\begin{abstract}
We construct many nonpolytopal nonsimplicial Gorenstein$^{*}$ meet
semi-lattices with nonnegative toric
$g$-vector, supporting a conjecture of Stanley.  These are formed as Bier spheres
over the face posets of multiplexes, polytopes constructed by Bisztriczky
as generalizations of simplices.
\end{abstract}

\section{Introduction}
A poset $P$ with a minimum $\hat{0}$ is called  \emph{Gorenstein$^{*}$} if the order complex of $P-\hat{0}$ is Gorenstein$^{*}$ (that is, if we add a maximum to $P$, then $\hat{P}=P\cup \{\hat{1}\}$ is both Eulerian and Cohen-Macaulay). For $P$ a Gorenstein$^{*}$ poset, denote its toric $g$-vector by $g(P)$. Our starting point is the following conjecture of Stanley.

\begin{conjecture}(\cite{Stanley:GeneralizedH-vectors}, Conjecture 4.2.(c,d))\label{conj:StanleyToric-g}
Let $P$ be a Gorenstein$^{*}$ meet semi-lattice. Then $g(P)$ is an $M$-sequence. In particular, $g(P)$ is nonnegative.
\end{conjecture}

Conjecture \ref{conj:StanleyToric-g} is known to hold in the following cases: face posets of the boundary complexes of simplicial polytopes \cite{Stanley:NumberFacesSimplicialPolytope-80}, strongly-edge-decomposible spheres \cite{Babson-Nevo, Murai-EdgeDecomposible} (these include Kalai's squeezed spheres \cite{Kalai-manyspheres, Murai-EdgeDecomposible}), Bier spheres \cite{BierPosets-BPSZ} - all these cases are simplicial spheres, and the nonnegativity of $g(P)$ is known for the boundary poset of arbitrary polytopes \cite{Karu:HLforPolytopes}.

In this work we construct many nonpolytopal nonsimplicial Gorenstein$^{*}$ meet semi-lattices with nonnegative toric
$g$-vector, supporting Conjecture \ref{conj:StanleyToric-g}. Specifically,
\begin{theorem}\label{thm:main}
Let $n\geq d\geq 4$ be integers and $M$ be the boundary poset of the $d$-dimensional multiplex with $n+1$ vertices. Let $I$ be an ideal in $M$ which contains all the elements in $M$ of rank at most $\lceil\frac{d}{2}\rceil+1$. Let $P$ be the Bier poset $P=\Bier(M,I)$. Then $P$ is a Gorenstein$^{*}$ meet semi-lattice and $g(P)$ is nonnegative.
\end{theorem}

%To prove this theorem, we introduce a generalization of bistellar moves, from simplicial complexes to posets, and analyze their effect on the toric $g$-vector.

In Section \ref{sec:prelim} we give the needed background on $g$-vectors, Bier posets and multiplexes; in Section \ref{sec:cumpute-g} we prove Theorem \ref{thm:main} by analyzing the effect on the toric $g$-vector of a local move on Gorenstien$^*$ posets; in Section \ref{sec:last} we remark on nonpolytopality and nonsimpliciality of $\Bier(M,I)$.

\section{Preliminaries}\label{sec:prelim}
\textbf{The $g$-vector.}
We follow Stanley \cite{Stanley:GeneralizedH-vectors}.
Let $P$ be a Gorenstein$^{*}$ poset of rank $r$, or more generally, a lower Eulerian poset (i.e., a graded poset such that all of its intervals are Eulerian posets, that is, for each interval, the number of elements of even degree equals the number of elements of odd degree), and define recursively the following polynomials in $x$, $g(P,x)$, $h(P,x)$ and $f(P,x)$:
$h(P,x)=x^{r}f(P,\frac{1}{x})$,
$f(\emptyset,x)=g(\emptyset,x)=1$,  for $f(P,x)=h_0+h_1x+h_2x^2+...$ let
$$g(P,x)=h_0+(h_1-h_0)x+...+(h_{\left\lfloor\frac{r}{2}\right\rfloor}-h_{\left\lfloor\frac{r}{2}\right\rfloor-1})x^{\left\lfloor\frac{r}{2}\right\rfloor},$$ and define
$$f(P,x)=\sum_{t\in P} g([\hat{0},t),x)(x-1)^{r-r(t)}$$
where $r(t)$ is the rank of $t$ and for $w,y\in P$
$$[w,y):=\{z\in P:\ w\leq z <y\}.$$

When adding a maximal element to $P$ results in an Eulerian poset,
we have $f(P,x)=h(P,x)$.
If $P$ is the face poset of an Eulerian simplicial complex $\Delta$ then $g(P)$ coincides with the simplicial $g$-vector $g(\Delta)$.

For $P$ the face poset of the boundary of a polytope $P'$, $h(P)$ encodes the dimensions of the combinatorial intersection homology associated with $P'$ \cite{BBFK-2002,BBFK-2005,Bressler-Lunts}, and a hard-Lefschetz type theorem for this module shows the nonnegativity of $g(P)$ \cite{Karu:HLforPolytopes}.

For a poset $P$ denote by $\widehat{P}$ the poset obtained by adding to it a maximum $\hat{1}$. If $P$ has a maximum $\hat{1}$, let $\partial P := P - \{\hat{1}\}$.

For posets $P$ and $Q$ let $P*Q$ be their product poset, i.e. $(w,y)\leq (w',y')$ in $P*Q$ iff  $w\leq w'$ in $P$ and
$y\leq y'$ in $Q$. We need the following facts, stated by Kalai \cite{Kalai:NewBasis} for polytopes.
\begin{observation}\label{obs:Kalai}
Let $P_1$ and $P_2$ be Eulerian posets. Then

(a) $g(\partial(P_1*P_2),x)=g(\partial P_1,x)\cdot g(\partial P_2,x)$, and

(b) $h(\partial P_1 * \partial P_2,x)=h(\partial P_1,x) \cdot h(\partial P_2,x)$.
\end{observation}
For completeness, we include their proofs.
\begin{proof}
We prove part (a) by induction on $r':=r(\partial(P_1*P_2))$. The case $r'\leq 1$ is trivial. For part (b) we prove more generally for lower Eulerian posets $Q_1,Q_2$ that
$f(Q_1 * Q_2,x)=f(Q_1,x) \cdot f(Q_2,x)$, by induction on $r'':=r(Q_1*Q_2)$. Call this assertion (b$''$). The case $r'' \leq 1$ is trivial.

First we prove (b$''$) for $r''$ assuming (a) is true for all $r'<r''$:
\begin{equation*}
\begin{split}
f(&Q_1 * Q_2,x)=\\
&\sum_{F_1\in Q_1, F_2\in Q_2}g(\partial([\hat{0},F_1]*[\hat{0},F_2]))(x-1)^{r(Q_1)+r(Q_2) -r(F_1) -r(F_2)}\\
&=
(\sum_{F_1\in Q_1}g([\hat{0},F_1))(x-1)^{r(Q_1)-r(F_1)})(\sum_{F_2\in Q_2}g([\hat{0},F_2))(x-1)^{r(Q_2)-r(F_2)})\\
&\qquad\qquad= f(Q_1,x)f(Q_2,x),
\end{split}
\end{equation*}
where the second equality follows by induction from part (a).
Thus, (b$''$) follows, hence also (b).

Next we prove (a) for $r'=l$ assuming (a) is true for $r'<l$.
Each lower interval of $\partial(P_1*P_2)$ satisfies exactly one of the following 3 possibilities - it contains either $P_1$, or $P_2$, or neither. Thus, using the recursive definition of $f(P,x)$ and the induction hypothesis we get
\begin{equation*}
\begin{split}
f(\partial(P_1*P_2),x)= g(\partial P_1,x)f(\partial P_2,x)
&+ g(\partial P_2,x)f(\partial P_1,x)\\
&+ (x-1)f(\partial P_1,x)f(\partial P_2,x)
.\end{split}
\end{equation*}

Let $T_r$ be the operator that truncates a polynomial at degree $r$. Thus, if $P$ is an Eulerian poset of rank $r$, $g(P,x)=T_{\left\lfloor\frac{r}{2}\right\rfloor}((1-x)f(P,x))$. Clearly,
the polynomial $\bar{g}(P,x):=(1-x)f(P,x)-g(P,x)$ satisfies that all its nonzero terms have degree larger than $\left\lfloor\frac{r}{2}\right\rfloor$. Thus,
\begin{equation*}
\begin{split}
(1-x)f(\partial(P_1*P_2),&x) =
g(\partial P_1,x) [g(\partial P_2,x)+\bar{g}(\partial P_2,x)]\\
&\qquad+
g(\partial P_2,x)[g(\partial P_1,x)+\bar{g}(\partial P_1,x)]\\
& -
[g(\partial P_1,x)+\bar{g}(\partial P_1,x)][g(\partial P_2,x)+\bar{g}(\partial P_2,x)] \\
&\quad=
g(\partial P_1,x) g(\partial P_2,x) -
\bar{g}(\partial P_1,x) \bar{g}(\partial P_2,x),
\end{split}
\end{equation*}
and we conclude that
\begin{equation*}
\begin{split}
g(\partial(P_1*P_2),x)&=T_{\left\lfloor\frac{l}{2}\right\rfloor}(g(\partial P_1,x) g(\partial P_2,x) -
\bar{g}(\partial P_1,x) \bar{g}(\partial P_2,x))\\
& =  g(\partial P_1,x) g(\partial P_2,x).
\end{split}
\end{equation*}
\end{proof}

%%%%THIS was wrong as induction cannot be used:
%Note that for any Eulerian poset $P$, $f(P,x)=(x-1)f(\partial P,x)+g(\partial P,x)$. Thus,
%$$h(P,x)=
%x^{r(P)}f(P,\frac{1}{x})=
%(1-x)h(\partial P,x)+ x^{r(P)}g(\partial P,\frac{1}{x})=
%g(\partial P,x)
%.$$
%Hence,
%$g(\partial(P_1*P_2),x)=h(P_1*P_2,x)=h(P_1,x)h(P_2,x)=g(\partial P_1,x)g(\partial P_2,x)$, where the second equality follows from (b"),
%and (a) follows.
%\end{proof}

\textbf{Bier posets.}
Bier posets were introduced in \cite{BierPosets-BPSZ}, generalizing the construction of Bier spheres, see e.g. \cite{Matousek-BU}. Here we slightly modify the definition in order to simplify notation.
Let $P$ be a finite poset with a minimum $\hat{0}$ and add to it a new element $\hat{1}$ greater than all elements of $P$ to obtain the poset $\hat{P}=P\cup \{\hat{1}\}$. Let $I$ be a proper ideal in $\hat{P}$. Then the poset $\Bier(P,I)$ consists of all intervals $[x,y]$ where $x\in I, y\in \hat{P}-I$, ordered by reverse inclusion.
If $\hat{P}$ is Eulerian then $\Bier(P,I)$ is Eulerian, and
the order complexes of the posets $P-\{\hat{0}\}$ and $\Bier(P,I)-\{[\hat{0},\hat{1}]\}$ are homeomorphic \cite{BierPosets-BPSZ}. In particular,

\begin{theorem}(\cite{BierPosets-BPSZ})\label{thm:BPSZ}
Let $I$ be an ideal in a poset $P$.
If $P$ is Gorenstein$^{*}$ then $\Bier(P,I)$ is Gorenstein$^{*}$.
\end{theorem}

In the case where $P$ is the face poset of the boundary of a simplex the posets $\Bier(P,I)$ are exactly the face posets of the Bier spheres, and each of their $g$-vectors is shown to be a Kruskal-Katona vector \cite{BierPosets-BPSZ}, and in particular an $M$-sequence, thus nonnegative. Note that the number of vertices in a $d$-dimensional Bier sphere is at most $2(d+2)$.
Note that for $I=P$, $\Bier(P,I)$ is isomorphic to $P$.

\textbf{Multiplex.}
Multiplexes were introduced by Bisztriczky \cite{Bisztriczky-Multiplex1} as a generalization of simplices.
For any integers $n\geq d\geq 2$ there exists a $d$-dimensional polytope on $n+1$ vertices $x_0,x_1,...,x_n$ with facets $$F_i:=\rm{conv}(x_{i-d+1},x_{i-d+2},...,x_{i-1},x_{i+1},x_{i+2},...,x_{i+d-1})$$ for $0\leq i\leq n$, with the convention that $x_j=x_0$ if $j<0$ and $x_j=x_n$ if $j>n$. Such polytope was constructed in \cite{Bisztriczky-Multiplex1}, is called a \emph{multiplex}, and we denote it by $M^{d,n}$. Note that $M^{d,d}$ is a simplex. We need the following known properties of multiplexes, shown by Bisztriczky \cite{Bisztriczky-Multiplex1} and  Bayer et. al. \cite{BayerBrueningStewart-Multiplex}.
\begin{theorem}\label{thm:multiplex}
(a) Every multiplex is self dual. (\cite{Bisztriczky-Multiplex1})

(b) If $F$ is a face in a multiplex $M$ then both $F$ and the quotient $M/F$ are multiplexes. (\cite{Bisztriczky-Multiplex1})

(c) $M^{d,n}$ has the same flag $f$-vector as the $(d-2)$-fold pyramid over the $(n-d+3)$-gon. Thus, for $P$ the boundary poset of $M^{d,n}$ we get $g(P,x)=1+(n-d)x$ and $h(P,x)=1+(n-d+1)x+...+(n-d+1)x^{d-1}+x^d$. (\cite{BayerBrueningStewart-Multiplex})
\end{theorem}

\section{Computing $g(\Bier(P,I))$}\label{sec:cumpute-g}
%\begin{definition}\label{def:PosetBistellar} \eran{CHANGE DEF??}
%Let $P$ be a pure poset with a minimum $\hat{0}$, $I$ an ideal in $P$ whose maximal elements have the same rank as $P$, and $J$ an ideal in $I$. Assume that the order complex of $I-\hat{0}$ is a ball whose boundary is the order complex of $J-\hat{0}$, and assume further that $I-J$ has a minimal element $t$. Assume that $J'\subseteq I' \subseteq P'$ and $t'$ satisfy the same conditions as $J\subseteq I \subseteq P$ and $t$, respectively.
%If $J'=J$ and $(P'-I')\cup J'=(P-I)\cup J$ (as posets) we say that $P'$ is obtained from $P$ by a \emph{poset bistellar move}.
%\end{definition}
%For example, if $P$ is the face poset of a simplicial sphere $\Delta$, $\sigma,\tau$ are simplices and $\sigma * \partial \tau$ is a full dimensional induced subcomplex of $\Delta$, and $P'$ is the face poset of $\Delta'=(\Delta -\sigma * \partial \tau)\cup (\partial\sigma * \tau)$, then $\Delta'$ is obtained from $\Delta$ by a (simplicial) bistellar move, and $P'$ is obtained from $P$ by a bistellar move.
%Here $J$ is the face poset of $\partial\sigma * \partial \tau$, $I$ of $\sigma * \partial \tau$, $I'$ of $\partial \sigma * \tau$, $t=\sigma$ and $t'=\tau$.

%Note that the order complexes of $P-\hat{0}$ and $P'-\hat{0}$ are homeomorphic.

We start by analyzing the effect of a certain local move on a pair $(P,I)$, of a Gorenstein$^*$ poset and an ideal in it, on the toric g-vector of $\Bier(P,I)$.
This local move can be thought of as a generalization of bistellar moves on simplicial spheres (however, this connection is not necessary for the topic of this paper and will not be stressed here).

Denote by $X^*$ the opposite of a poset $X$, that is its elements are $\{x^*:x\in X\}$ and $x<y$ in $X$ iff $y^*<x^*$ in $X^*$.

\begin{lemma}\label{lem:bistellar}
Let $P$ be a Gorenstein$^*$ poset, let $Q$ be an ideal in $P$, different from $\{\hat{0}\}$, and let $t$ be a maximal element in $Q$.
Then:
%
%(a) $\Bier(P,Q-\{t\})$ is obtained from $\Bier(P,Q)$ by a poset bistellar move.
\begin{equation*}
\begin{split}
h(\Bier(P,Q-\{t\}),&x)-h(\Bier(P,Q),x)\\
& = h([\hat{0},t),x)
g((t,\hat{1}]^*,x) - g([\hat{0},t),x)
h((t,\hat{1}]^*,x),
\end{split}
\end{equation*}
 where $\hat{1}$ is the maximum of $\hat{P}$.
 (In particular, $h(\Bier(P,Q-\{t\}),x)-h(\Bier(P,Q),x)$ is independent of the ideal $Q$ as long as $t$ is a maximal element in the ideal.)
\end{lemma}
\begin{proof}
%In Definition \ref{def:PosetBistellar} plug $I:=\{(y,z):\ y\in [\hat{0},t], z\in (t,\hat{1}] \}$, $I':=\{(y,z):\ y\in [\hat{0},t), z\in [t,\hat{1}] \}$ and $J=J'=\{(y,z):\ y\in [\hat{0},t), z\in (t,\hat{1}] \}$ and use Theorem \ref{thm:multiplex}(a,b) and the fact that a cone over a sphere is a ball to show (a).

By Theorem \ref{thm:BPSZ},
the h-vectors on the left are of Gorenstein$^*$ posets, hence are symmetric, thus
the left hand side equals
\begin{equation*}
f(\Bier (P,Q-\{t\}),x) - f( \Bier (P,Q),x),
\end{equation*}
which, by definition, equals
\begin{equation*}
\begin{split}
\sum_{z:t>z\in P} g([[\hat{0},\hat{1}],[z,t]),x)&(x-1)^{r(t)-r(z)}\\
&- \sum_{y:t<y\in P} g([[\hat{0},\hat{1}],[t,y]),x)(x-1)^{r(y)-r(t)}.
\end{split}
\end{equation*}
The last expression equals
\begin{equation*}
\begin{split}
\sum_{z\in [\hat{0},t)} g(\partial([\hat{0},z]*[t,\hat{1}]^*),x)&(x-1)^{r(t)-r(z)}\\
&- \sum_{y\in (t,\hat{1}]} g(\partial([\hat{0},t]*[y,\hat{1}]^*),x)(x-1)^{r(y)-r(t)}
,\end{split}
\end{equation*}
 and by Observation \ref{obs:Kalai} it equals
\begin{equation*}
\begin{split}
\sum_{z\in [\hat{0},t)} g([\hat{0},z),x)&g((t,\hat{1}]^*,x)(x-1)^{r(t)-r(z)}\\
&\qquad- \sum_{y\in (t,\hat{1}]} g([\hat{0},t),x)g((y,\hat{1}]^*,x)(x-1)^{r(y)-r(t)}
\\
&=
g((t,\hat{1}]^*,x)f([\hat{0},t),x) - g([\hat{0},t),x)f((t,\hat{1}]^*,x).
\end{split}
\end{equation*}
As the posets on the right hand side are Gorenstein$^*$, the assertion follows.
\end{proof}

%\section{Concluding remarks}
We now specialize to the case of a multiplex.
Let $M$ be the boundary poset of the multiplex $M^{d,n}$, $I\neq \{\hat{0}\}$ an ideal in $M$, $t$ a maximal element in $I$, and denote $\Delta h:=h(\Bier(M,I-\{t\}),x)-h(\Bier(M,I),x)$
and
$$\Delta g:=g(\Bier(M,I-\{t\}),x)- g(\Bier(M,I),x).
$$

\begin{lemma}\label{lem:g(M)-bistellar}
With notation as above, if $d>r(t)> \lceil \frac{d}{2}\rceil +1$ then
\begin{equation*}
\begin{split}
\Delta g =g_1((t,\hat{1}]^*)x^{d-r(t)} + (g_1([\hat{0},t))g_1((t,\hat{1}]^*)&+1)x^{d-r(t)+1} \\
&+ g_1([\hat{0},t))x^{d-r(t)+2},
\end{split}
\end{equation*}
and hence is nonnegative, and if $r(t)=d>\lceil \frac{d}{2}\rceil$ (thus $t$ is a facet) then
$\Delta g= x+g_1([\hat{0},t))x^2$ and hence is nonnegative.
\end{lemma}
\begin{proof}
Denote $A=(t,\hat{1}]^*$ and $B=[\hat{0},t)$. They are the boundary posets of multiplexes by Theorem \ref{thm:multiplex}(a,b).
By Lemma \ref{lem:bistellar} and Theorem \ref{thm:multiplex}(c) we get for $r(t)<d$:
\begin{equation*}
\begin{split}
\Delta h=& (1+h_1(B)x+...+h_1(B)x^{r(t)-2}+x^{r(t)-1})(1+g_1(A)x)\\
&- (1+g_1(B)x)(1+h_1(A)x+...+h_1(A)x^{d-r(t)-1}+x^{d-r(t)}).
\end{split}
\end{equation*}
The coefficient of $x^j$ in $\Delta h$, denoted by $[x^j]$, equals $h_{j-1}(B)g_1(A)+h_j(B) - h_j(A) - g_1(B)h_{j-1}(A)$, with the convention $h_{-1}=0$. Thus, assuming  $r(t)<d$, for $j\leq \left\lfloor\frac{d}{2}\right\rfloor$ we get:
$[x^j]=0$ if $0\leq j<d-r(t)$,
$[x^{d-r(t)}]=g_1(A)$,
$[x^{d-r(t)+1}]=1+g_1(A)h_1(B)$,
and $[x^j]=h_1(A)h_1(B)$ for $d-r(t)+1<j\leq \left\lfloor\frac{d}{2} \right\rfloor$.
Thus, $\Delta g$ is as claimed for $r(t)<d$, and in particular is nonnegative by Theorem \ref{thm:multiplex}(c).

In case $r(t)=d$ we get that $\Delta h = x+h_1(B)x^2+...+h_1(B)x^{d-2}+x^{d-1}$ and $\Delta g= x+g_1(B)x^2$ is nonnegative.
\end{proof}

The conditions imposed on $r(t)$ in Lemma \ref{lem:g(M)-bistellar} imply that $d\geq 4$, which is the interesting range for Conjecture \ref{conj:StanleyToric-g}.

\begin{proof}
[Proof of Theorem \ref{thm:main}]
By repeated application of Lemma \ref{lem:g(M)-bistellar} we conclude the nonnegativity of $g(P)$ in Theorem \ref{thm:main}. By Theorem \ref{thm:BPSZ}, $P$ is Gorenstein$^*$, and clearly $P$ is a meet semi-lattice: let $\wedge_M, \wedge_P$ be the meet operations in $M,P$ respectively, and $\vee_M$ the join operation in $M$. Then $[a,b], [c,d]\in P$ satisfy $[a,b]\wedge_P[c,d] = [a\wedge_M c, b\vee_M d]$.
\end{proof}

%\begin{corollary}\label{cor:g(Bier)>=0}
%Let $P$ be the boundary poset of the multiplex $M^{d,n}$ with $d>5$, $I\neq \{\hat{0}\}$ an ideal in $P$ such that all the maximal elements in $I$ have rank at least $\lceil \frac{d}{2}\rceil +1$. Then $g(\Bier(P,I),x)$ is nonnegative.
%\end{corollary}

In fact, Lemma \ref{lem:g(M)-bistellar} can be simplified, as we will show below. The inequality $g_2(M)\geq g_1([\hat{0},t))g_1((t,\hat{1}]^*)$ (see \cite{Braden-MacPherson:Kalai} for the general statement for polytopes, and e.g. \cite{Braden:Remarks} for the validity for arbitrary polytopes) together with $g_2(M)=0$ shows that at least one of $g_1([\hat{0},t)), g_1((t,\hat{1}]^*)$ is zero.
Here is a simple combinatorial criterion to tell which.
Bisztriczky \cite[Lemmas 3,4]{Bisztriczky-Multiplex1} described the $4$-gons among the $2$-faces of a multiplex. In particular:
\begin{observation}\label{obs:g1squares}
For a multiplex $M$, $g_1(M)$ is the number of $4$-gons among the $2$-dimensional faces of $M$.
\end{observation}
%\begin{proof}
%We perform the following straightforward computation. Let $f_{\triangle}$ and $f_{\Box}$ be the numbers of $3$-gons and $4$-gons, respectively, in the $2$-skeleton of $M$.
%Then:
%\end{proof}

Thus, combining with Lemma \ref{lem:g(M)-bistellar}, we obtain the following.

\begin{proposition}\label{prop:toric-g-by-Squares}
Let $M$ and $I$ be as in Theorem \ref{thm:main}. Denote by $\Box^*$ the set of faces  $t$ in $M$ such that  $[t,\hat{1}]^*$ is the face poset of a $4$-gon.
Then,
\begin{equation}\label{eqBox}
\begin{split}
g(\Bier(M,I),x)=&
\sum_{t\in \hat{M}-I}x^{r(\hat{M})-r(t)} +\
g_1(M) x\\
&+\sum_{t\in M-I} g_1([\hat{0},t))x^{r(\hat{M})-r(t)+1}\\
&\quad+\sum_{t\in M-I}\sum_{s\in \Box^*}\delta_{t\subseteq s} x^{r(\hat{M})-r(t)-1}
\end{split}
\end{equation}
\hfill$\square$
\end{proposition}

First, let us consider the special case where $M$ is the boundary of a simplex. Then all summands but the first one on the right hand side of (\ref{eqBox}) vanish ($\Box^*=\emptyset$ in this case). Now $(\hat{M}-I)^*$ is an ideal in the Boolean lattice dual to $\hat{M}$, thus it corresponds to the face poset of a simplicial complex, whose $f$-polynomial is equal, by Proposition \ref{prop:toric-g-by-Squares}, to $g(\Bier(M,I),x)$.
For a different proof see \cite{BierPosets-BPSZ}.
In particular, Conjecture \ref{conj:StanleyToric-g} holds in this case.

Whether for all multiplexes $M$ and ideals $I$ the $g(\Bier(M,I))$ are $M$-sequences is still open.

Next, we show that for any $M,I,P$ as in Theorem \ref{thm:main}, $$(1,g_1(P),g_2(P))$$
 is a Kruskal-Katona vector (in particular, an $M$-sequence).
By self-duality of polytopes $(\hat{M}-I)^*$ is the face poset of a polyhedral complex $\Gamma$, whose $f$-vector corresponds to the first summand on the right hand side of (\ref{eqBox}). Denote by $t^*$ the face in $\Gamma$ corresponding to $t\in \hat{M}-I$.
Take the disjoint union of $\Gamma$ with a totally ordered set $V$ of $g_1(M)$ vertices, and for any face $t^*\in \Gamma$ add as faces the cones over $t^*$ having apex each of the first $g_1([\hat{0},t))$ vertices in $V$.
As $t \subseteq s$ implies $g_1([\hat{0},t))\leq g_1([\hat{0},s))$, this results in a polyhedral complex, denoted by $\Delta$, whose $f$-polynomial is given by the sum in the first two lines of the right hand side of (\ref{eqBox}).
(We remark that any polyhedral complex
is a meet semi-lattice with the \emph{diamond property},
namely, any interval of rank $2$ has $4$ elements, and thus, by a result of Wegner \cite{Wegner}, there is a simplicial complex with the same $f$-vector as $\Delta$.)
To the $1$-skeleton of $\Delta$ one can add a diagonal at each $4$-gon face $t^*$ where $t\in (M-I)\cap \Box^*$, and by (\ref{eqBox}) the resulted graph $G$ has $f$-vector $(1,g_1(P),g_2(P))$.

%%%%%%%%%%%%%%%%%%%%%%%%%%%%%%%%
\section{Concluding remarks}\label{sec:last}
We keep the notation of Theorem \ref{thm:main}.

Note that $M^{d,n}$ has non-simplex $2$-faces iff $n>d$, in this case they are of the form $F=\rm{conv}(x_i,x_{i+1},x_{i+d},x_{i+d+1})$ \cite{Bisztriczky-Multiplex1}, and note that $I$ contains all rank $3$ elements in $M$. Thus, $\Bier(M,I)$ is non-simplicial - for example $[F,\hat{1}]$ is a non-simplex $2$-face.

We now argue that $\Bier(M,I)$ is not polytopal for many choices of $I$.
$M$ has $\binom{d+1}{i+1}+(n-d) \binom{d-1}{i}$ $i$-faces (this was computed already in \cite{Bisztriczky-Multiplex1}, and follows also from Theorem \ref{thm:multiplex}(c)).
Let $F$ be a facet of $M$. Then $F$ is $(d-1)$-dimensional with at most $2(d-1)$ vertices. For $d\geq 6$, a simple computation shows that the number of  ideals $I_F$ in $[\hat{0},F)$ which contain all the elements in $[\hat{0},F)$ of rank at most $\lceil\frac{d}{2}\rceil+1$ is at least $2^{\binom{d}{\lceil  \frac{d}{2} \rceil +2}}$,
thus the number of non-isomorphic such $\Bier(F,I_F)$ is at least
$\frac{2^{\binom{d}{\lceil \frac{d}{2}\rceil +2}}}{(4(d-1))!}=\Omega(2^{(1-\epsilon_d)2^d/\sqrt{d}})$ where $\epsilon_d\rightarrow 0$ as $d\rightarrow \infty$ (see e.g. \cite{BierPosets-BPSZ}).
It is known that the number of non-isomorphic polytopes with at most $4(d-1)$ vertices is less than $2^{64d^3+O(d^2)}$ \cite{GoodmanPollack:fewPolytopes-86,Alon:fewPolytopes}, hence most of the $\Bier(F,I_F)$ as above are nonpolytopal for $d>>1$. Fix an ideal $I_F$ such that $\Bier(F,I_F)$ is nonpolytopal.
For an ideal $I$ in $P$ such that $I\cap [\hat{0},F)=I_F$ note that $\Bier(F,I_F)$ is an interval in
$\Bier(M,I)$, hence $\Bier(M,I)$ is nonpolytopal.
\\

Next, we discuss shellability of $\Bier(M,I)$, for any ideal $I$. The order complex of $\Bier(M,I)$ is obtained from the barycentric subdivision of the boundary complex of $M$ by a sequence of edge subdivisions \cite[Theorem 2.2]{BierPosets-BPSZ}, thus is polytopal and in particular shellable. In \cite[Theorem 4.1]{BierPosets-BPSZ} it is shown that if $M$ is a simplex then $\Bier(M,I)$ itself is shellable (in this case $\Bier(M,I)$ is the face poset of a simplicial sphere). A generalization of shellability to regular finite CW complexes was defined in \cite{Bjorner84:PosetsRegularCWBruhatOrder} and was shown to be equivalent to CL-shellability of its  augmented dual \cite[Proposition 4.2]{Bjorner84:PosetsRegularCWBruhatOrder}. We now exhibit an EL-labeling of $Q=\Bier(M,I)^*\cup\hat{1}$ by the labels poset $[0,n]:=0<1<\cdots<n$, in particular proving its CL-shellability, and thus the shellability of $\Bier(M,I)$.

Denote by $E(Q)$ the cover relations $a\prec b$ in $Q$.
Recall that a map $c: E(Q)\rightarrow [0,n]$ is an \emph{EL-labeling}
if for any interval $[a,b]$ in $Q$ the following holds \cite{BW:LexShellPosets}: 

(i) there is a unique \emph{rising} unrefinable chain $a=a_0\prec a_1\prec\cdots \prec a_t=b$, i.e. $c(a_0\prec a_1)<\cdots < c(a_{t-1}\prec a_t)$ in $[0,n]$, and 

(ii) among the label-sets of unrefinable chains from $a$ to $b$ the set $\{c(a_0\prec a_1),\cdots,c(a_{t-1}\prec a_t)\}$ is the lexicographically least.

\begin{theorem}
$Q=\Bier(M,I)^*\cup\{\hat{1}\}$ has an EL-labeling. Thus, $\Bier(M,I)$ is shellable.
\end{theorem}
\begin{proof}
First we define an EL-labeling $c_M$ of $M$ itself. For a facet
$F_i(M)=F_i:=\rm{conv}(x_{i-d+1},x_{i-d+2},...,x_{i-1},x_{i+1},x_{i+2},...,x_{i+d-1})$ of $M$ let $c_M(F_i\prec M)=i$ for $i\neq 0,n$, and $c_M(F_0\prec M)=d$, $c_M(F_n\prec M)=n-d$. Each $d'$-face $N$ in $M$ with vertex set $J$ is a multiplex with the induced order on $J$ from $[0,n]$, and we apply the rule above to define $c_M(F,N)$ for the facets $F$ of $N$. Namely, if $J=\{v_0<v_1<\cdots <v_l\}$, any facet
$F$ is of the form $F_j(N)$ for a unique $0\leq j\leq l$, and define $c_M(F_0(N)\prec N)=v_{d'}$, $c_M(F_l(N)\prec N)=v_{l-d'}$ and $c_M(F_j(N)\prec N)=v_{j}$ for $0<j<l$.
(Note that we do not distinguish between $i$ and $x_i$ as labels; essentialy the set of labels is $\{ 0,1,\cdots,n \}$.)

\begin{lemma}\label{claim:EL(M)}
$c_M$ is an EL-labeling of $M$.
\end{lemma}
\begin{proof}
As $c_M(x\prec y)\in y\setminus x$ all the labels in a chain are distinct. 
First we show that (i) and (ii) in the definition of EL-labeling hold for intervals of the form $[\hat{0},N]$. Let $J$ be the vertex set of $N$, as above. Clearly, the lexicographically least $d'$-subset of $J$ is $\{v_0,v_1,\cdots,v_{d'}\}$, and is attained as a label-set by the rising unrefinable chain $\hat{0}\prec\cdots\prec F_0(F_0(N))\prec F_0(N)\prec N$. We now show there is no other rising maximal chain in $[\hat{0},N]$. Else, by induction on the rank of $N$, w.l.o.g. we can assume that in such a rising chain $C$ the facet of $N$ taken is some $F=F_j(N)$ with $j>0$. Then, there are less then $d'-1$ vertices in $F$ smaller than $c_M(F\prec N)$, so $C$ is not rising, a contradiction.

For a general interval $[x,y]$ in $M$, $x=\cap_{j\in J}F'_j$,  the intersection runs over all the facets of $y$ containing $x$, and we use notation so that $J$ is the set of vertices $y\setminus x$. Let $M(J)$ be the $(r(y)-r(x)-1)$-dimensional multiplex on the ordered set $J$, where $r(x)$ is the rank of $x$ in the poset $M$.
By duality, $[x,y]\cong M(J)\cong M(J)$ with isomorphisms $F'_j\mapsto j\mapsto j^*$, where $j^*$ is the facet $F_j(M(J))$, so
$c_M(F'_j\prec y)=c_{M(J)}(j^*\prec \hat{1})$. Thus, by the case of lower intervals considered before, $c_M$ satisfies (i) and (ii) for \emph{any} interval $[x,y]$ in $M$.
\end{proof}

Back to $Q$, denote $\hat{1}:=[\hat{1}_M,\hat{1}_M]\in Q$. 
For a cover relation $[x,y]\prec[x',y']$ in $Q$ we define the label $c([x,y]\prec[x',y'])$ as follows: if $x=x'$ then $y\prec y'$ and define $c([x,y]\prec[x',y'])=c_M(y\prec y')$; else $y=y'$ and $x'\prec x$ and define $c([x,y]\prec[x',y'])=c_M(x'\prec x)$. 

For any interval $[[x,y],[x',y']]$ in $Q$, we have shown in Lemma \ref{claim:EL(M)} that there are unique rising maximal chains $C_x,C_y$ in $[x',x],[y,y']$ in $M$, with label-sequences $L_x,L_y$ resp.
There is a unique way to shuffle $L_x$ and $L_y$ in order to get an increasing label-sequence $L$, and this shuffle induces a rising maximal chain $C$ in $[[x,y],[x',y']]$. As $L_x,L_y$ are lexicographically least in the corresponding intervals, so is $L$ in $[[x,y],[x',y']]$.
Moreover, any rising maximal chain $C'$ in $[[x,y],[x',y']]$ with label set $L'$ must decompose as $L'=L_x\cup L_y$ and thus $C'=C$.
\end{proof}

%\eran{More directions for improvement??:
%(1) Conclude that $g(Bier(M,I))$ is also an $M$-sequence. Here, when arguing by removing $t$ from the ideal, for $r(t)=d$ claim is clear, alse the problem is to show that there is enough room to add $g_1((t,\hat{1}]^*)$ to $[x^{r-d(t)}]$ (the contributions to  $[x^{r-d(t)+1}]$ and  $[x^{r-d(t)+2}]$ are easy to handle).
%(2) Assume $g(P)\geq 0$ for Gorestein* $P$ (also intervals), let $I$ be an ideal in $P$, show $g(Bier(P,I)\geq 0$.
%(3) More suffisticated local moves in the last Remark.
%}

\section*{Acknowledgments}
We would like to thank the anonymous referees for their helpful comments.

\bibliography{gbiblio}
\bibliographystyle{plain}
\end{document}